\documentclass[paper=a4, fontsize=12pt]{scrartcl} 
\usepackage[a4paper, total={5.5in, 10in}]{geometry}
\usepackage[iso-8859-7]{inputenc} 
\usepackage[english]{babel} 
\usepackage{amsmath,amsfonts,amsthm} 
\usepackage{times}
\usepackage{color}
\usepackage{lipsum} 
\usepackage[shortlabels]{enumitem}
\usepackage{sectsty} 
\allsectionsfont{  \normalfont\mdseries \textbf} 

\usepackage[noadjust]{cite}
\usepackage[pdftex]{graphicx}
\usepackage[font={small,it}]{caption}
\usepackage{url}
\usepackage{epsfig}
\usepackage{empheq}
\usepackage{epstopdf}
\usepackage{pstricks}
\usepackage{float}
\numberwithin{equation}{section} 
\numberwithin{figure}{section} 
\numberwithin{table}{section} 

\newcommand{\horrule}[1]{\rule{\linewidth}{#1}} 

\title{
\normalfont \normalsize
\horrule{2pt} \\[0.2cm] 
\LARGE  On the  Representation of \\ Involutive Jamesian Functions \\
\horrule{2pt} \\[0.0cm] 
}

\author{\large{Nikos Stamatis\footnote{ School of Applied Mathematical and Physical Sciences, National Technical University of Athens, Greece;  \texttt{nstam84@gmail.com}}}} 

\date{\normalsize\today} 

\newcommand{\ee}{\varepsilon}
 \newcommand{\repr}{  \mathcal{J}_{\text{rep}}}
  \newcommand{\invo}{  \mathcal{J}_{\text{inv}}}

\theoremstyle{definition}\newtheorem{definition}{Definition}[section]
\theoremstyle{plain}\newtheorem{proposition}[definition]{Proposition}
\theoremstyle{definition}\newtheorem{lemma}[definition]{Lemma}

\newtheorem{corollary}[definition]{Corollary}

\theoremstyle{plain}\newtheorem{theorem}[definition]{Theorem}
\theoremstyle{definition}\newtheorem{question}[definition]{Open
Problem}

\newtheorem*{quest}{Question}
\theoremstyle{definition}\newtheorem{remark}[definition]{Remark}
\newcommand\textlcsc[1]{\textsc{\MakeLowercase{#1}}}

\begin{document}

\maketitle




\begin{quote}\footnotesize{\textbf{Abstract.}}
 Involutive Jamesian Functions are functions aimed  to
predict the outcome of an athletic competition. They were introduced
in 1981 by Bill James, but until recently little was known regarding
their form.    Using methods from quasigroup theory we  are able to
obtain a complete description  of them.
\end{quote}

\section{Introduction}
In 1981, Bill James \cite{james} introduced the notion of a function
$P$, now called the James Function, in his effort to answer the
following question: \emph{Suppose that two baseball teams $A$ and
$B$ with winning percentages equal to $a$ and $b$ respectively, play
against each other. What is the probability $P(a,b)$ that $A$ beats
$B$?}

Instead of just providing a formula for $P(a,b)$, James was also
interested in what properties   such a function should satisfy. In
this spirit, he proposed the following self-evident conditions, now
called the \emph{proto-James conditions}:

\begin{enumerate} \item $P(a,a)=\frac{1}{2}$ and   $P(a,\frac{1}{2})=a$, for every $a\in
(0,1)$.
\item If $a>b$, then $P(a,b)>\frac{1}{2}$ and if $a<b$, then
$P(a,b)<\frac{1}{2}$.
\item If $b<\frac{1}{2}$, then $P(a,b)>a$ and if $b>\frac{1}{2}$,
then $P(a, b)<a$.
\item $0\leq P(a,b)\leq 1$ and if $ a\in (0,1)$, then $P(a, 0)=1$ and
$P(a,1)=0$.
\item $P(a,b)+P(b,a)=1$.
\end{enumerate}
After introducing the desired conditions, he   provided an example
of a function satisfying them, the explicit formula of which was
given by Dallas Adams:
\begin{equation}
P(a,b)=\frac{a(1-b)}{a(1-b)+(1-a)b}. \label{adams}
\end{equation}
He then conjectured that this was the only function that satisfied
his conditions.

In 2015, Christopher  Hammond, Warren Johnson and Steven Miller
published a survey \cite{hjm} on James functions and, among other
things, they addressed this conjecture. They discovered a whole
class of functions satisfying the proto-James conditions, thus
disproving the James conjecture. To avoid working with pathological
examples, they proposed the following axioms:
\begin{definition}\label{def1}
A function $J: [0,1]\times [0,1] \setminus\{(0,0),
(1,1)\}\rightarrow [0,1]$ is called an \emph{involutive Jamesian
function} if
\begin{enumerate}\item $J\big(a, J(a,b)\big)=b$ for every $a, b\in
(0,1)$,  \ \ \text{(involutive property)}
\item $J(a, b)+J(b, a)=1$ for every $a, b\in
(0,1)$ and
\item for a fixed $b_0\in (0,1)$, the function $J(\cdot, b_0)$ is   strictly
increasing and for  $b_0\in [0,1]$, it is nondecreasing.
\end{enumerate}
\end{definition}

\noindent Clearly, every involutive Jamesian function also satisfies
the proto-James conditions. They showed that even if we restrict
ourselves to this class of well behaved functions, we can still find
a wealth of counterexamples: For every increasing homeomorphism
$f:(0,1)\rightarrow \mathbb{R}$ such that $f(1-a)=-f(a)$ for every
$a \in (0,1)$, the function
\begin{equation}
J(a, b)= f^{-1}\big(f(a)-f(b)\big), \label{f-rep form}
\end{equation}
is an involutive Jamesian function and James' original example is
just a special case of (\ref{f-rep form}).  They closed their survey
proposing three open problems, one of which was whether the converse
were also true:

\begin{question} Is it true that every involutive Jamesian function
can be written in the form (\ref{f-rep form}) for some function $f$?
\end{question}

\noindent In what follows, an involutive Jamesian function which can
be written in this  form will be called \emph{$f$-representable}, or
just \emph{representable}. The main purpose of this paper is to
prove that there exist involutive Jamesian functions which are not
$f$-representable. This is achieved in Paragraph \ref{Jamesian Loops
Paragraph} with the introduction of the notion of Jamesian Loops.
Although our proof is existential,  we also show how we can adopt it
in order to produce explicit examples. We close our paper by
providing a complete characterization of involutive
Jamesian functions. As it turns out,  these functions can indeed be represented in a form very similar to (\ref{f-rep form}). 


\section{Basic Properties of Involutive Jamesian Functions}
First, we   recall  a few basic properties. Their proofs can be
found in \cite{hjm}.

\begin{proposition}Let $J$ be an involutive Jamesian function. Then
\begin{enumerate}
\item $J(a, b)=c$ if and only if $J(a, c)=b$. This is actually
equivalent to the involutive property.
\item $J(a, b)=J(1-b, 1-a)$ for every $a, b \in (0,1)$.
\item $J(a, a)=\tfrac{1}{2}$ and $J(a, \tfrac{1}{2})=a$ for every
$a\in (0,1)$.
\item $J$ is continuous.
\end{enumerate}
\end{proposition}

Although it is not necessary, it will prove useful to change the
definition of involutive Jamesian functions slightly. The two
definitions are equivalent, so this is more a matter of convenience
than of essence.
\begin{definition}\label{def2}
A function $J: (0,1)\times (0,1)\rightarrow (0,1)$ is called an
\emph{involutive Jamesian function} if
\begin{enumerate}\item $J\big(a, J(a,b)\big)=b$ for every $a, b\in
(0,1)$,  \ \ \text{(involutive property)}
\item $J(a, b)+J(b, a)=1$ for every $a, b\in
(0,1)$ and
\item for a fixed $b_0\in (0,1)$, the function $J(\cdot, b_0)$ is   strictly increasing.
\end{enumerate}
\end{definition}

Clearly, for every $J$ satisfying the properties of Definition
\ref{def1}, its restriction on $(0,1)\times (0,1)$ also satisfies
Definition  \ref{def2}. Conversely, suppose that $J$ satisfies
Definition \ref{def2} and let $a\in (0,1)$.  We first show that if
  $(b_n)_{n\in \mathbb{N}}$ is a strictly decreasing sequence in
  $(0,1)$
with $b_n\rightarrow 0$, then $J(a, b_n)\rightarrow 1$.

Suppose not. The sequence ($J(a, b_n))_{n\in \mathbb{N}}$ is
strictly increasing and bounded above, so it has to converge to its
supremum. Suppose that $c_n=J(a, b_n)\rightarrow c_0<1.$ Then
$J(c_n, a)=1-b_n$ is a strictly increasing sequence converging to
$1$.  We pick a $c_0' \in (c_0, 1)$. Since $J(\cdot, a)$ is strictly
increasing and $c_n<c_0<c_0'$, we have that
\begin{equation}
J(c_n, a)<J(c_0, a)<J(c_0', a)
\end{equation}
for every $n$ and since $J(c_n,a)\rightarrow 1$, this implies that
$1\leq J(c_0, a)<J(c_0', a),$ a contradiction. This argument, which
is just an adaptation of the proof of \cite[Proposition 7]{hjm} to
our context,   implies that $J$ can be extended continuously and
uniquely on $[0,1]\times [0,1] \setminus\{(0,0), (1,1)\}$ and this
extension satisfies Definition \ref{def1}.

The reason why it is more convenient to work with the second
definition  will become clear soon, but before that, it is important
to begin our study  by  noting that the $f$-representability of
Jamesian functions has a very useful reformulation:
\begin{proposition}\label{transitivityEquivalentFormulations} Let $J$ be an involutive Jamesian function. The
following conditions are equivalent: \begin{enumerate}
\item $J\big(J(a, c), J(b, c)\big)=J(a, b)$ for every
$a, b, c \in (0,1)$, \ \ \ \ \ \text{(transitivity)}
\item $J\big(J(a, b), J(a, c)\big)=J(c, b)$ for every
$a, b, c \in (0,1)$,
\item $J\big(b, J(a, c)\big)=J\big(c, J(a, b)\big)$ for every
$a, b, c \in (0,1)$,
\item $J$ is $f$-representable.
\end{enumerate}
\end{proposition}
\begin{proof}
The equivalence of (1), (2) and (3) is a straightforward application
of the involutive property. For example, to obtain (3) from (2) one
needs to plug $b:=J(a, b)$ into (2), while for the converse, to pick
an $x_0\in (0,1)$ such that $b=J(a, x_0)$ and plug it in (3) to
obtain $J\big(J(a, b), J(a, c)\big)= J\big(x_0, J(a,c)\big)=J\big(c,
J(a, x_0)\big)=J(c, b)$.

If $J$ is $f$-representable, then
\begin{eqnarray*}
J\big(J(a, c), J(b, c)\big)&=& f^{-1}\big(f(J(a, c))-f(J(b, c))\big)
\\
&=& f^{-1}\big(f(a)-f(c) -f(b) +f(c)\big) \\
&=&f^{-1}\big(f(a)  -f(b)  \big) \\&=&J(a, b),
\end{eqnarray*}
so it satisfies (1). Only the direction (1) $\Rightarrow$ (4) is
   non-trivial and it is actually a Theorem by M. Hossz\'{u}
   \cite{hos}.

\end{proof}

\begin{corollary}
An involutive Jamesian function $J$ is $f$-representable if and only
if it is transitive, that is $J\big(J(a, c), J(b, c)\big)=J(a, b)$
for every $a, b, c \in (0,1)$. \end{corollary}

\noindent Although we will not prove   Hossz\'{u}'s theorem in
detail, it is crucial  to understand the main idea behind its proof,
as our approach on the representation of Jamesian functions relies,
essentially,  on the   same principles.

Hossz\'{u} attempted to solve the functional equation of
transitivity
\begin{equation} F\big(F(x, t), F(y, t)\big)=F(x, y), \ \ \forall x, y, t\in
(0,1), \label{transitivityFE}
\end{equation}
assuming   that $F:(0,1)\times (0,1)\rightarrow (0,1)$ was a
continuous and strictly monotonic function. He observed that if a
function $F$ satisfied (\ref{transitivityFE}), then the operation
$x\cdot y=F(x, 1-y)$ defined a continuous group   on $(0,1)$. A
theorem by Brouwer \cite{bro} asserts that all one dimensional
continuous groups are isomorphic to the additive group of real
numbers, so the operation $\cdot$ on $(0,1)$ had to be of the form
\begin{equation}
x\cdot y= f^{-1}\big(f(x)+f(y)\big), \  \ \forall x, y \in (0,1),
\end{equation}
with $f: (0,1)\rightarrow \mathbb{R}$   a  homeomorphism. As
\begin{equation}F(x,y)=x\cdot(1-y)=f^{-1}\big(f(x)+f(1-y)\big), \end{equation} he was then
able to give a complete description of the solutions of
(\ref{transitivityFE}).

\begin{theorem} {\textbf{\normalfont(\textbf{M. Hossz\'{u}, 1953).}}} A
continuous and strictly monotonic function $F: (a,b)\times (a, b)
\rightarrow (a,b)$ satisfies the functional equation of transitivity
\begin{equation}
 F\big(F(x, t), F(y, t)\big)= F(x, y), \ \forall x, y, t \in (a, b),
\end{equation}
if and only if $F$ is written in the form
\begin{equation}
F(x, y)= f^{-1}\big(f(x)-f(y)\big),  \ \ \text{(quasi-difference)}
\end{equation}
where $f$ is a continuous and strictly monotonic function.
\end{theorem}

To return to our problem, it is easy to see that an involutive
Jamesian function $J$ is transitive if and only if the induced
operation $a\cdot b = J(a, 1-b)$ forms a group (Proposition
\ref{correspondence between loops and James functions}). Since we do
not know yet whether such functions are necessarily transitive, we
cannot hope that the corresponding operations will form a group.
However, as we will show in the next section, they form  the next
best structure we could hope for: a \emph{loop}.

\section{Loops and Jamesian Loops}

Quasigroups  are non-associative algebraic structures that  arise
naturally in many   Mathematical fields. Although they first
appeared in the literature in the early 1900s,  it was not  until
the 1930s that they were defined and studied systematically.
Quasigroups with an identity element are called \emph{loops}. They
emerged by the influential works of Ruth Moufang (1935) and Gerrit
Bol (1937) and soon they became a separate area of research.
For a fascinating  historical overview of quasigroups and loops, the
reader should check \cite{pfl1}.

Before we start exploring the connection between loops and Jamesian
functions, we need a few basic definitions and properties from loop
theory. A standard reference for their algebraic theory   is
\cite{pfl2}, whereas the  first chapter of \cite{ns} offers an
introduction to their topological aspects.

\begin{definition}
A \emph{loop} $(L, \cdot)$ is a  set $L$ equipped with an operation
$\cdot: L\times L \rightarrow L$, satisfying the following
properties:
\begin{enumerate}  \item There exists an element $e\in L$,  such
that $a\cdot e=e\cdot a=a$ for every $a\in L$. This element is
called the identity element, or the unit of $L$. \item For every
$a\in L$ there exists an $a^{-1}\in L$ such that $a\cdot
a^{-1}=a^{-1} \cdot a=e$.
\end{enumerate}
\end{definition}

\noindent So a loop is a structure that satisfies every group axiom
except, possibly, for   associativity. A loop which is not a group
is called a \emph{proper loop}.

\begin{definition} A loop $L$ is said to have the \emph{inverse property } (or
\emph{IP}) if \begin{equation}x(x^{-1}y)=y=(yx^{-1})x  \text{ for
every } x, y\in L. \label{IP property}\end{equation}
\end{definition}
\noindent It is easy to check that a loop with the inverse property
also satisfies $(xy)^{-1}=y^{-1}x^{-1}$ for every $x, y \in L$: Let
$a, b\in L$. We substitute $x:=(ab)^{-1}$ and $y=b^{-1}$ into
(\ref{IP property}) to get
$(ab)^{-1}\left((ab)b^{-1})\right)=b^{-1}$. By the inverse property,
$(ab)b^{-1}=a$, so $(ab)^{-1}\cdot a=b^{-1}$, which implies that
$(ab)^{-1}=b^{-1}a^{-1}$.

\begin{definition} A \emph{topological loop} $(L, \cdot)$ is a loop which
is also a topological space, such that the product and the inverse
operations are continuous.
\end{definition}

\begin{definition} Two topological loops $(L, \cdot)$, $(L', \ast)$ are said to
be \emph{  isomorphic} if there exists a (topological) homeomorphism
$f: L\rightarrow L'$ such that     $f(a\cdot b)=f(a)\ast f(b)$ for
every $a, b \in L$.
\end{definition}

\begin{remark}\label{propernessremark} We can easily see that if two loops are isomorphic and one of
them is proper, then the other one has to be proper as well: Say
that $(L, \cdot)$ is proper. Then there exist $a, b, c\in L$ such
that $a\cdot(b\cdot c) \neq (a\cdot b) \cdot c$. Since $f$ is one to
one,

\begin{equation*}f\big(a\cdot(b\cdot c)\big) =f(a)\ast \big(f(b)\ast f(c)\big)\neq
\big(f(a)\ast f(b)\big)\ast f(c)=f\big((a\cdot b) \cdot
c\big),\end{equation*} which shows that $\ast$ is not associative.
Similarly one can check that commutativity and IP are properties
   preserved by isomorphisms.
\end{remark}

An isomorphism from a topological loop $L$ to itself is called an
\emph{automorphism}. To every loop we associate two important
families of automorphisms, its left and right translations:
\begin{definition}Given a loop
$(L, \cdot)$ and $b_0\in L$, the function  $R_{b_0}: L\rightarrow
L,$ defined as $R_{b_0}(a)=a\cdot b_0$ for $ a \in L$, is called a
\emph{right translation.} Similarly, given an $a_0 \in L$, the
function $L_{a_0}(b)=a_0\cdot b$ is called a \emph{left
translation}. If the loop is commutative, the sets of left and right
translations coincide.
\end{definition}

The group $G$ generated by the left translations, equipped with the
Arens topology, plays an important role in the classification of
loops. In our work we will not need this approach, but the curious
reader should consult \cite{ns} for more details.

\subsection{Jamesian Loops} \label{Jamesian Loops Paragraph}
We now return to our main question:

\begin{quest}  Given an involutive Jamesian function $J$, what properties
does the operation $a\cdot b=J(a, 1-b)$ defined on $(0,1)$ satisfy?
\end{quest}

First of all, for every $a, b$, $a\cdot b=J(a, 1-b)=J(b, 1-a)=b\cdot
a,$ so it has to be commutative. Additionally, $a\cdot
\tfrac{1}{2}=J(a, \tfrac{1}{2})=a$ and $a\cdot (1-a)=J(a,
a)=\tfrac{1}{2}$, so $\big( (0,1), \cdot\big)$ is in fact a loop,
having $e=\tfrac{1}{2}$ as an identity and such that the inverse of
every $a$ is $a^{-1}=1-a$. The involution property of $J$ can be
rewritten as
\begin{equation*}
b=J\big(a, J(a, b)\big) = J\big( a, 1-J(b,a)\big) = a\cdot J(b,a)= a
\cdot\big(b\cdot(1-a)\big)
\end{equation*}
 for every $a, b$, which, because of the commutativity, is just the inverse property of
loops. Similarly, the identity $J(a, b)+J(b,a)=1$ implies that
$(a\cdot b)^{-1}=b^{-1}\cdot a^{-1}$. This   already follows from
the inverse property, so we   may omit it when we define the notion
of Jamesian loops.

 Finally, since $J$ is strictly increasing with respect to
the first variable, the left (and right) translations of our
operation need to be strictly increasing functions. We collect all
these observations into the definition of  Jamesian loops:

\begin{definition}\label{defJamesLoops} A topological loop $\big((0,1), \cdot\big)$ is
called a \emph{Jamesian loop} if
\begin{enumerate}  \item  it is commutative, \item its unit is $e=\tfrac{1}{2}$ and the inverse of every $a\in (0,1)$ is $a^{-1}=1-a$, \item
$a\big(b(a^{-1})\big)=b$ for every $a, b\in (0,1)$ and
\item the right (and left) translations are strictly increasing
functions.
\end{enumerate}
\end{definition}
\noindent As we expected, there is a one to one correspondence
between involutive Jamesian functions and Jamesian loops:

\begin{proposition} \label{correspondence between loops and James functions} If $J$ is an
involutive Jamesian function then the operation $\cdot$ for which
$a\cdot b=J(a,1-b)$ for $a, b \in (0,1)$, is a Jamesian loop.
Conversely, if $\big((0,1),\cdot\big)$ is a Jamesian loop then the
function $J(a, b)=a\cdot(1-b)$ is an involutive Jamesian function.
In addition, a Jamesian loop is a group if and only if the induced
involutive Jamesian function is $f$-representable.
\end{proposition}
\begin{proof}
From the preceding discussion, if $J$ is an involutive Jamesian
function, then $\cdot$ is a Jamesian loop. Conversely, if
$\big((0,1),\cdot\big)$ is a Jamesian loop and we define $J(a,
b)=a(1-b)$, then the inverse property of $\cdot$ implies that $J(a,
J(a, b))=b$ for every $a, b$. Similarly, $(a\cdot
b)^{-1}=b^{-1}\cdot a^{-1}$ implies that $J(a, b)+J(b,a ) =1$. The
function $J(\cdot, b_0)$ is strictly increasing as $J(\cdot,
b_0)=R_{b_0}$ and every right translation is strictly increasing. So
$J$ satisfies all three axioms of involutive Jamesian functions.

For the additional part, the induced operation is always a loop, so
we are in fact seeking an equivalent characterization of
associativity. Let $a, b, c\in (0,1)$. We compute the expressions
\begin{eqnarray*}
(a\cdot b)\cdot c &=& J(a, 1-b)\cdot c = J\big(J(a, 1-b), 1-c\big)\\
&=&J\big(c, 1-J(a, 1-b)\big)\\& =&J\big(c, J( 1-b,a)\big) \ \ \ \ \ \  \text{and}\\
a\cdot(b\cdot c) &=& a\cdot J(b, 1-c)= J\big(a, 1-J(b,1-c)\big) \\
&=&J\big(a,  J(1-c,b)\big) \\
&=&J\big(a,  J(1-b,c)\big).
\end{eqnarray*}
The operation is associative if and only if $J\big(c, J(
b,a)\big)=J\big(a,  J(b,c)\big)$ for every $a, b, c \in (0,1)$,
which  is equivalent to the $f$-representability of $J$ due to
Proposition \ref{transitivityEquivalentFormulations}.
\end{proof}

%

The following example was constructed  by Salzmann in his attempt to
``\emph{prove the existence of  planar associative division neorings
which possess an additive loop with the inverse property and which
are homeomorphic to $\mathbb{R}$}'' \cite[pg. 135]{str}. With a
slight modification it also proves the existence of proper Jamesian
loops and, in view of the previous proposition,  of involutive
Jamesian functions which are not
$f$-representable. 

\begin{theorem} \label{SalzTheorem} {\normalfont{\textbf{(H. Salzmann, 1957).}}} There exists a commutative proper topological loop $(\mathbb{R}, \cdot)$ with the inverse
property having  $e=0$ as an identity  and such that the inverse of
every $x\in \mathbb{R}$ is $x^{-1}=-x$. Additionally, its left (and
right) translations are strictly increasing.
\end{theorem}
\begin{proof}
The operation, as defined in \cite[pg. 459]{sal}, is:
\[
    x\ast t =\begin{cases}
x+\frac{1}{2}t, &\text{ if } \frac{x}{t}\in (-\infty, -\frac{3}{2}]\cup [1,+\infty), \\
\frac{1}{2}x+t, &\text{ if } \frac{x}{t}\in  [-\frac{2}{3}, 1] , \\
2x+2t,&\text{ if } \frac{x}{t}\in  [-\frac{3}{2}, -\frac{2}{3}],
\\
x, & \text{ if } t=0.
\end{cases}
  \]
  Clearly,  $e=0$ is an identity element for the $\ast$-operation
  and one can easily check that $x\ast y= y\ast x$ for every $x,
  y\in \mathbb{R}$. Additionally, since for every $x\neq 0$,
  $\tfrac{x}{-x}=-1$, then
  $x\ast (-x)= 2x+ (-2x)=0$ and the pair $(\mathbb{R}, \ast)$ is a commutative loop such that every $x$ has $-x$ as its inverse.
With a little more effort we can show that it also possesses  the
inverse property:

\[
    (x\ast t)\ast (-t) =\begin{cases}
(x\ast t)-\frac{1}{2}t, &\text{ if } \frac{x\ast t}{-t}\in (-\infty, -\frac{3}{2}]\cup [1,+\infty), \\
\frac{1}{2}(x\ast t)-t, &\text{ if } \frac{x\ast t}{-t}\in  [-\frac{2}{3}, 1] , \\
2(x\ast t)-2t,&\text{ if } \frac{x\ast t}{-t}\in  [-\frac{3}{2},
-\frac{2}{3}],
\\
x, & \text{ if } t=0.
\end{cases}
  \]

  \begin{itemize} \item If $\tfrac{x}{t}\in (-\infty, -\frac{3}{2}]\cup
  [1,+\infty)$, then $x\ast t=x+\tfrac{1}{2}t$ and $\tfrac{x\ast
  t}{-t}=-\tfrac{x}{t}-\tfrac{1}{2}$, so $\tfrac{x\ast t}{-t} \leq
  -\tfrac{3}{2}$ or $\tfrac{x\ast t}{-t}\geq 1.$ This yields that $(x\ast t)\ast
  (-t)=x\ast t-\tfrac{1}{2}t=x+\tfrac{1}{2}t-\tfrac{1}{2}t=x. $

\item  If $\tfrac{x}{t}\in
  [-\tfrac{2}{3},1]$, then $x\ast t= \tfrac{1}{2}x+t$ and $\tfrac{x\ast
  t}{-t}= -\tfrac{x}{2t}-1$, so $\tfrac{x\ast t}{-t} \in [-\tfrac{3}{2}, -\tfrac{2}{3}]$ and $(x\ast t)\ast
  (-t)=2(x\ast t)-2t=2(\tfrac{1}{2}x+t)-2t=x.$

  \item   If $\tfrac{x}{t}\in
  [-\tfrac{3}{2},-\tfrac{2}{3}]$, then $x\ast t= 2x+2t$ and $\tfrac{x\ast
  t}{-t}= -\tfrac{2x}{t}-2$, so $\tfrac{x\ast t}{-t} \in [-\tfrac{2}{3}, 1]$ and $(x\ast t)\ast
  (-t)=\tfrac{1}{2}(x\ast t)-t=\tfrac{1}{2}(2x+2t)-t=x.$
  \end{itemize}

\noindent In any case $(x\ast t)\ast (-t)=x$.

On the other hand,  $(\mathbb{R}, \ast)$ is not associative. We will
actually show that it is not even power associative, by proving that
$x^2\ast x^2\neq x^3\ast x$ for every $x\neq 0$. Obviously both $x$
and $x^2=x\ast x$ are uniquely determined and since the operation is
commutative, $x^3=x\ast(x\ast x)=(x\ast x)\ast x$ is also defined.
However,  this is not true  for $x^4$:

  \begin{eqnarray*}
x^2&=&x\ast x= x+\tfrac{1}{2}x=\tfrac{3}{2}x, \\
x^3&=&x\ast(x^2)=x\ast \tfrac{3}{2}x= \tfrac{1}{2}x+\tfrac{3}{2}x=2x, \\
x^2\ast x^2&=&\tfrac{3}{2}x \ast \tfrac{3}{2}x =\tfrac{3}{2}x
+\tfrac{3}{4}x =\tfrac{9}{4}x, \\
x^3\ast x &=& (2x) \ast x =2x+\tfrac{1}{2}x=\tfrac{5}{2}x.
  \end{eqnarray*}
  Since $$x^2\ast x^2=(x\ast x) \ast
(x\ast x)\neq (x\ast (x\ast x))\ast x=x^3\ast x,$$  the operation is
not associative, thus $(\mathbb{R}, \ast)$ is a proper loop.
\end{proof}

All that is left now is to transfer Salzmann's example from
$\mathbb{R}$ to  $(0,1)$ preserving its main properties.
\begin{lemma} \label{transferSalzmann}
Let $(\mathbb{R}, \ast)$ be a topological loop satisfying all the
aforementioned  properties of Salzmann's example. That is,
$(\mathbb{R}, \ast)$
 is a commutative proper topological loop with the inverse property,
 its unit is $e=0$, the inverse of each
$x\in\mathbb{R}$ is $x^{-1}=-x$ and   its left translations are
strictly increasing. Then for every strictly increasing
homeomorphism $f: (0, 1)\rightarrow \mathbb{R}$ such that
$f(1-x)=-f(x)$, the operation $\cdot$ for which
\begin{equation} a\cdot b = f^{-1}\big(f(a)\ast f(b)\big), \ \ \ a,
b\in (0,1),
\end{equation} defines a proper Jamesian loop on $(0,1)$.
\end{lemma}
\begin{proof}
The proof follows easily from the fact that the loops $(\mathbb{R},
\ast)$ and $\big((0,1), \cdot\big)$ are  isomorphic under the
function $g=f^{-1}$. Since $f(\tfrac{1}{2})=0$, $g(0)=\tfrac{1}{2}$
is the unit of $(0,1)$. Additionally,
$$x\cdot(1-x)=f^{-1}\big(f(x)\ast f(1-x)\big)=f^{-1}\big(f(x)\ast
(-f(x)\big)=f^{-1}(0)=\tfrac{1}{2}, $$  which shows that the inverse
operation on $(0,1)$ is given by $x^{-1}=1-x$. Commutativity,
properness and IP are properties preserved by isomorphisms (see
Remark \ref{propernessremark}), so they must hold on $\big((0,1),
\cdot\big)$ as well. The left translations being strictly increasing
follows from the fact that $f$ is strictly increasing.
\end{proof}

\begin{corollary} There exist involutive Jamesian functions which
are not   representable.
\end{corollary}
\begin{proof}
By   Proposition \ref{correspondence between loops and James
functions}, every proper Jamesian loop induces such a function and
proper Jamesian loops do exist, as was shown in Theorem
\ref{SalzTheorem} and Lemma \ref{transferSalzmann}.
\end{proof}

\begin{remark}
Actually Lemma \ref{transferSalzmann} can be used to produce an
infinity of counterexamples: Let $f, g: (0,1)\rightarrow \mathbb{R}$
be distinct homeomorphisms with the property that $f(1-x)=-f(x)$ and
$g(1-x)=-g(x)$ for every $x \in (0,1)$. Let also $J$ be Salzmann's
loop on $\mathbb{R}$ and $J_f$, $J_g$ the corresponding proper
Jamesian loops on $(0,1)$, as defined in Lemma
\ref{transferSalzmann}. We will show that $J_f$ and $J_g$ are
distinct. As $f$ and $g$ are continuous and not identically equal,
there exists an $a_0\in (\tfrac{1}{2}, \infty)$ and an
$\varepsilon>0$, such that $f(x)\neq g(x)$ for every $x\in
I=(a_0-\varepsilon, a_0+\varepsilon)$. Without loss of generality we
may also assume that $f(x)>g(x)$ on $I$. Let $x_0 \in I$. Then
\begin{equation*}
J_f(x_0, 1-x_0)=f^{-1}\big(f(x_0)\ast
f(x_0)\big)>g^{-1}\big(g(x_0)\ast g(x_0)\big) =J_g(x_0, 1-x_0),
\end{equation*}
as all $\ast$, $f$ and $g$ are strictly increasing. In
particular, $J_f\not\equiv J_g$.


Similarly, if one begins with two non-isomorphic Jamesian loops,
then their corresponding Jamesian functions are distinct. To sum up,
there are two different ways of constructing distinct families of
involutive Jamesian functions. One is, for a fixed Jamesian loop on
$\mathbb{R}$, to let the homeomorphism $f:(0,1)\rightarrow
\mathbb{R}$ of Lemma \ref{transferSalzmann} vary. The other is to
consider Jamesian functions induced by non-isomorphic Jamesian
loops.
\end{remark}

\begin{remark}
Even though there exist non-representable involutive Jamesian
functions, it is very natural to ask whether   these
``pathological'' functions can at least be  approximated by
representable ones. Let us describe the setting first: We denote the
classes of representable Jamesian functions and involutive Jamesian
functions by $\mathcal{J}_{\text{rep}}$ and
$\mathcal{J}_{\text{inv}}$ respectively and as we have shown,
$\mathcal{J}_{\text{rep}}$ is a proper subset of $
\mathcal{J}_{\text{inv}}$. Let also $S=(0,1)\times (0,1)$ be the
open unit square of the plane.

Since both $\mathcal{J}_{\text{rep}}$ and $
\mathcal{J}_{\text{inv}}$ are subsets of   $(C_b(S),
\|\cdot\|_\infty)$, the Banach space of bounded and continuous real
valued functions on $S$ equipped with the supremum norm, we can view
them as metric spaces with the induced metric. It is easy to check
that $\mathcal{J}_{\text{rep}}$ is not dense in $
\mathcal{J}_{\text{inv}}$ and, in fact, a continuity argument and an
appeal to
 Propositions \ref{transitivityEquivalentFormulations} and
  \ref{correspondence between
loops and James functions}  shows that none of the elements of
$\invo\setminus \repr$ can be approximated by elements of $\repr$.
\end{remark}

Before closing this section it is worth to mention a question
regarding Jamesian loops  and concerns what
happens in the case of differentiability. 
Salzmann's loop is continuous, but clearly   not differentiable
everywhere, so the following remains open: \begin{question} Is it
true that every differentiable involutive Jamesian function is
 representable? \end{question}

\subsection{A Specific Example} \label{Specific Example Paragraph}
In the previous paragraph we provided a rule to construct
non-representable Jamesian  functions, however,  until now, we have
not been able to give an explicit example of one. In this paragraph
we will try to do so.

By Lemma \ref{transferSalzmann}, in order to find a formula for
$J(a, b)$, we need to specify a homeomorphism $f: (0,1)\rightarrow
\mathbb{R}$ and then compute
\begin{equation}\label{explicit formula}
  J(a, b) = f^{-1}\big(f(a)\ast f(1-b)\big),
\end{equation}
where $(\mathbb{R}, \ast)$ is   Salzmann's loop. If we examine the
definition of $\ast$, we see that the plane is partitioned into
three regions and $a\ast b$ is defined according to which region the
number $\tfrac{a}{b}$ belongs to. So, in order to compute
(\ref{explicit formula}), first we have to determine where
$\tfrac{f(a)}{f(1-b)}$ belongs, for every $a, b\in (0,1)$. Depending
on $f$, this can be a difficult to nearly impossible task. 

It soon becomes clear that we need to avoid  this complexity by
picking $f$ appropriately: If $f$ acts as the identity function in a
neighborhood around $\tfrac{1}{2}$, then $\tfrac{f(a)}{f(1-b)}$ is
easily computed for $a, b$ in this neighborhood and, quite possibly,
so is $f^{-1} \big(f(a)\ast f(1-b)\big)$. This gives us a simple
formula for $J(a, b)$. 
The downside is that on $(0,1)\setminus I$, $f$ is no longer the
identity, but this is not something we can avoid, as any
homeomorphism $f:(0,1)\rightarrow \mathbb{R}$ is bound to deform
some regions of $(0,1)$.

In what follows, we  construct an example of a non-representable
involutive Jamesian function $J$   and we give an explicit formula
for $J(a, b)$, but only for $(a, b)\in \Delta \subseteq
S=(0,1)\times(0,1)$, where $\Delta$ is a region we will specify
using the ideas we just discussed.

Let $\ee>0$ and set $I_\ee=[\ee,1-\ee]$. We define
$f:I_\ee\rightarrow \mathbb{R}$ by $f(x)=x-\tfrac{1}{2}$ for every
$x\in I_\ee$. Clearly $f(1-x)=-f(x)$ for every $x\in I_\ee$ and f
can be extended to a homeomorphism from $(0,1)$ to $\mathbb{R}$
satisfying the same property. We will denote this extension again by
$f$. The inverse of $f$ has the property that
$f^{-1}(x)=x+\tfrac{1}{2}$ for every $x\in I_\ee'=[\ee-\tfrac{1}{2},
\tfrac{1}{2}-\ee] \subseteq \mathbb{R}$. If   $a, b\in I_\ee$ such
that $f(a)\ast f(b)= (a-\tfrac{1}{2})\ast (\tfrac{1}{2}-b)\in
I_\ee'$, then
\begin{equation}
J(a, b)= f^{-1}\big(f(a)\ast f(1-b)\big)= f(a)\ast
f(1-b)+\frac{1}{2},
\end{equation}
so we have a simple formula for $J(a, b)$. This is possible on the
following subsets of $\mathbb{R}^2$:
\begin{eqnarray*}
A_1&=&\left\{(a, b)\in S: \tfrac{a-1/2}{1/2-b}\in \left(-\infty,
-\tfrac{3}{2}\right]\cup [1,\infty)  \text{ and }
a-\tfrac{b}{2}-\tfrac{1}{4} \in I_\ee'\right\}, \\
A_2&=&\left\{(a, b)\in S: \tfrac{a-1/2}{1/2-b}\in
\left[-\tfrac{2}{3}, 1\right] \text{ and }
\tfrac{a}{2}-b+\tfrac{1}{4} \in I_\ee'\right\}, \\
A_3&=&\left\{(a, b)\in S: \tfrac{a-1/2}{1/2-b}\in
\left[-\tfrac{3}{2},-\tfrac{2}{3} \right] \text{ and } 2a-2b \in
I_\ee'\right\}.
\end{eqnarray*}
So we can explicitly define $J$ on $\Delta=A_1\cup A_2\cup A_3$ 
\begin{figure}[htb]
\centering \def\svgwidth{220pt}
\scalebox{0.8}{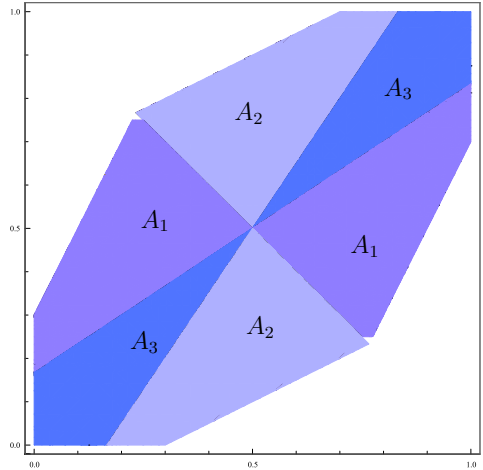} \caption{The region
$\Delta=A_1\cup A_2\cup A_3$ where $J$ can be  defined  explicitly.
}\label{graf}
\end{figure}

\noindent and the formula of $J$ is:
\[
 J(a, b)=\left(a-\tfrac{1}{2}\right)\ast \left(\tfrac{1}{2}-b\right)+\tfrac{1}{2} =\begin{cases}
a-\frac{b}{2}+\frac{1}{4}, &\text{ if } (a, b)\in A_1, \\
\frac{a}{2}-b+\frac{3}{4}, &\text{ if } (a,b)\in A_2, \\
2a-2b+\frac{1}{2},&\text{ if } (a, b)\in A_3,
\\
f^{-1}\big(f(a)\ast f(1-b)\big), &\text {everywhere else.}
\end{cases} \]

\begin{figure}
  \centering
    \scalebox{0.3}{\includegraphics{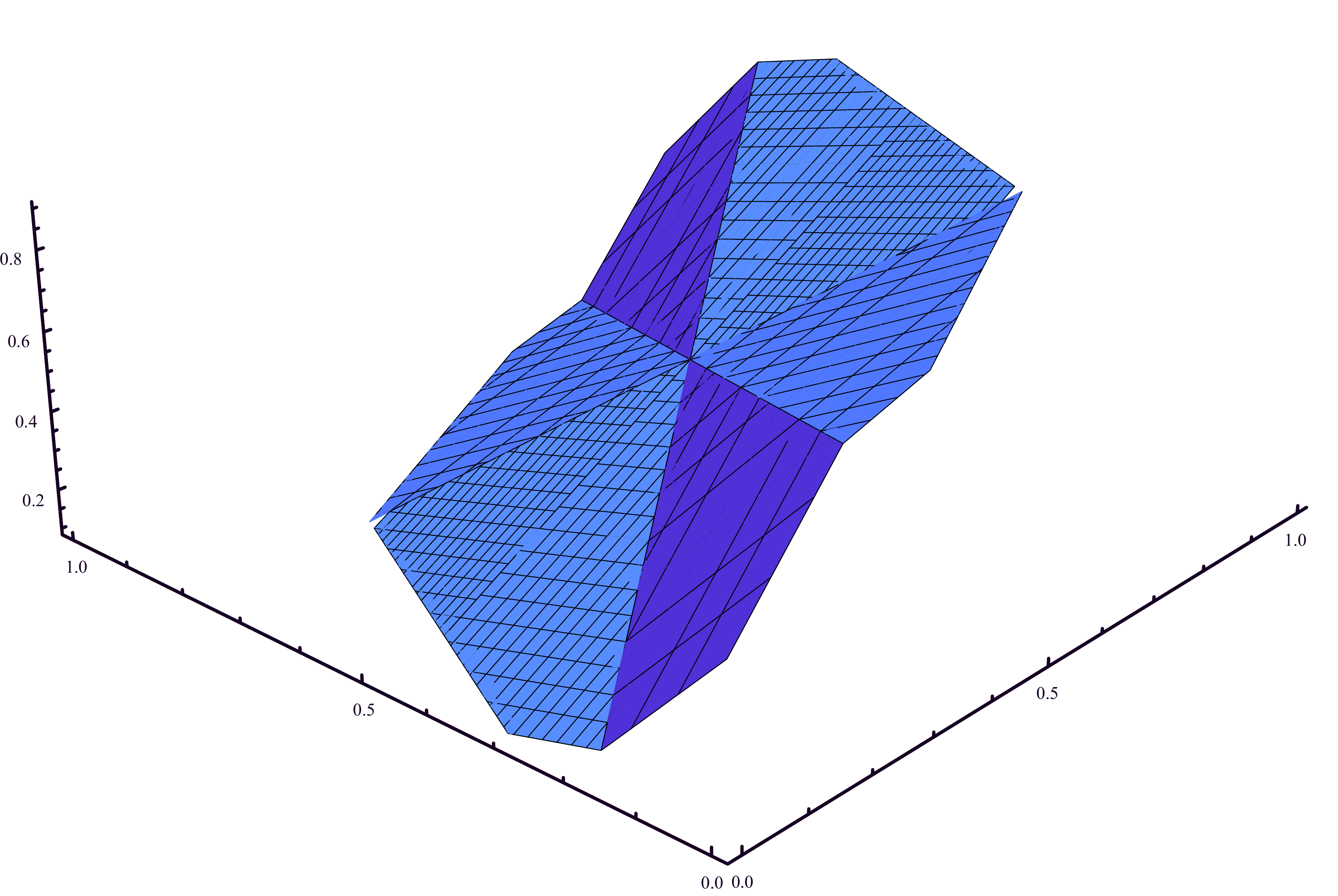}}
    \caption{The graph of $J$ on $\Delta = A_1\cup A_2\cup A_3$.  }\label{graf2}
    \end{figure}

\section{Conclusions}
We proved that an involutive Jamesian function may not necessarily
be written in the form \begin{equation} J(a, b)=f^{-1}\big(f(a)+
f(1-b)\big). \end{equation}
Despite this, we can still   find a very similar representation by
allowing   the addition in the expression ``$f(a)+f(1-b)$'' to be
replaced by a more general operation $\ast$:
\begin{proposition} Every involutive Jamesian function $J$ can be
written as \begin{equation}J(a, b)=f^{-1}\big(f(a)\ast f(1-b)\big),
\end{equation} where $f:(0,1)\rightarrow \mathbb{R}$ is a strictly
increasing homeomorphism with $f(1-x)=-f(x)$ for every $x\in (0,1)$
and $(\mathbb{R}, \ast)$ is a loop.
\end{proposition}

This gives us a complete description of involutive Jamesian
functions, but at a certain cost: The  well studied and understood
operation of real addition has been replaced by a mysterious and
probably complicated real loop operation. There is a deep theory
behind the attempts of classification of loops on $\mathbb{R}$ and
only certain classes of them  have been completely classified.
Unfortunately the Jamesian loops do not seem to belong in any of
these classes, so the (equivalent) problems of classifying
involutive Jamesian functions and classifying Jamesian loops remain
open. For an idea of how complicated, but also exciting, this
problem can be, the reader should check \cite[Chapter~18]{ns}.

\paragraph*{\footnotesize{\textbf{Acknowledgments:}}}
 \footnotesize{{The author is grateful to Professors Christopher Hammond,
Warren Johnson and Steven Miller for their communication and for the
valuable comments and suggestions they made, after reading an early
version of this manuscript.}

\sectionfont{\bfseries  \small \textlcsc }

\end{document}